\documentclass[11pt]{article} 
\usepackage{amssymb, latexsym}
\newtheorem{defin}{}
\newtheorem{saetze}[defin]{}
\newtheorem{lemmas}[defin]{}
\newtheorem{folger}[defin]{}
\newtheorem{bemerk}[defin]{}
\newtheorem{exampl}[defin]{}

\newenvironment{theorem}  {\begin{saetze}\it {\bf Theorem:}}{\end{saetze}}
\newenvironment{lemma}    {\begin{lemmas}\it {\bf Lemma:}}{\end{lemmas}}

\newenvironment{exam}     {\begin{exampl}\it {\bf Example:}}{\end{exampl}}
\newenvironment{proof}    {{\it Proof}:}{{\hfill \fillbox \bigskip}}

\newcommand{\fillbox}{\mbox{$\bullet$}}
\newcommand{\ra}{\rightarrow}
\newcommand{\Ra}{\Rightarrow}

\newcommand{\ms}{\mapsto}

\newcommand{\N}{\mathbb N}
\newcommand{\F}{\mathbb F}

\newcommand{\Q}{\mathbb Q}

\newcommand{\C}{\mathbb C}

\newcommand{\ideal}[1]{(#1)} % {\langle #1 \rangle}

\newenvironment{items}{\begin{list}{$\alph{item})$}
{\labelwidth18pt \leftmargin18pt \topsep3pt \itemsep1pt \parsep0pt}}
{\end{list}}

\newcommand{\bulit}{\item[$\bullet$]}

\begin{document}

\title{Deciding if a variety forms an algebraic group}
\author{John Abbott and Bettina Eick}
\date{\today}
\maketitle

\begin{abstract}
Let $n$ be a positive integer and let $f_1, \ldots, f_r$ be polynomials in 
$n^2$ indeterminates over an algebraically closed field $K$. We describe an
algorithm to decide if the invertible matrices contained in the variety of
$f_1, \ldots, f_r$ form a subgroup of $GL(n,K)$; that is, we show how to 
decide if the polynomials $f_1, \ldots, f_r$ define a linear algebraic group.
\end{abstract}

%%%%%%%%%%%%%%%%%%%%%%%%%%%%%%%%%%%%%%%%%%%%%%%%%%%%%%%%%%%%%%%%%%%%%%%%%%%%%
\section{Introduction}

Let $n$ be a positive integer and let $K$ be a field.
A linear algebraic group is a subgroup of $GL(n, K)$ defined by polynomial 
equations.  More precisely, for a linear algebraic group $G$ there exist 
polynomials $f_1, \ldots, f_r \in K[x_1, \ldots, x_{n^2}]$ such that 
$G = \{ g \in GL(n, K) \mid f_i(g) = 0 \mbox{ for } 1 \leq i \leq r \}$, 
where $f_i(g)$ is $f_i$ evaluated at the $n^2$ entries of the matrix $g$.  
Some well-known examples of algebraic groups are the general linear, the 
special linear, the orthogonal and the symplectic groups. 

Linear algebraic groups are studied extensively in the literature. But there
are only a few algorithms known to compute with such groups. Membership 
testing is easy by definition, but other algorithmic problems are often
rather difficult to decide. 

The aim of this note is to describe an algorithm to decide if a set of 
polynomials $f_1, \ldots, f_r$ in $n^2$ indeterminates over an algebraically 
closed field $K$ defines an algebraic group. More precisely, let 
$V(f_1, \ldots, f_r) = \{ v \in K^{n \times n} \mid f_1(v) = \ldots = 
f_r(v) = 0 \}$ denote the variety defined by the polynomials $f_1, \ldots, 
f_r$ and denote $V^*(f_1, \ldots, f_r) = V(f_1, \ldots, f_r) \cap GL(n, K)$.
Our algorithm decides if $V^*(f_1, \ldots, f_r)$ forms a group under matrix 
multiplication.

%%%%%%%%%%%%%%%%%%%%%%%%%%%%%%%%%%%%%%%%%%%%%%%%%%%%%%%%%%%%%%%%%%%%%%%%%%%%%
\section{Notation and Preliminaries}

Throughout, let $n$ be a positive integer, and $K$ be a field.  
Let $x = (x_1, \ldots, x_{n^2})$ be a list of commuting indeterminates.
For brevity, we shall write $K[x]$ to mean the polynomial ring 
$K[x_1, \ldots, x_{n^2}]$. Similarly, if $y = (y_1, \ldots, y_{n^2})$ 
is another list of commuting indeterminates, we shall write $K[x,y]$ to 
mean $K[x_1,\ldots,x_{n^2}, y_1,\ldots,y_{n^2}]$, and so on. 

We write $K^{n \times n}$ to denote the algebra of $n \times n$ matrices 
over $K$.  There is a natural isomorphism of vector spaces
\[ K^{n^2} \ra K^{n \times n} 
   :\quad (v_1, \ldots, v_{n^2}) \quad\ms\quad
       \left( \begin{array}{ccc}
        v_1 & \ldots & v_n \\
        \vdots &     & \vdots \\
        v_{n^2-n+1}     & \ldots & v_{n^2} 
       \end{array} \right ). \]
We shall frequently identify vectors in $K^{n^2}$ with matrices in 
$K^{n \times n}$ via this isomorphism. In particular, this isomorphism 
lets us evaluate a polynomial $f \in K[x]$ at a matrix $v \in 
K^{n \times n}$. Similarly, we can consider any point in $V(I)$ as a matrix.

%%%%%%%%%%%%%%%%%%%%%%%%%%%%%%%%%%%%%%%%%%%%%%%%%%%%%%%%%%%%%%%%%%%%%%%%%%%%%
\section{Varieties}
\label{extpro}
 
Let $K$ be a field, let $x = (x_1, \ldots, x_{n^2})$ be a list of commuting
indeterminates over $K$, let $f_1, \ldots, f_r \in K[x]$ and let $I = (f_1,
\ldots, f_r) \unlhd K[x]$. A first obstacle in computations with $V^*(I)$
is that this set is defined as an intersection of a variety and a group and 
this makes it difficult to apply methods from algebraic geometry or group 
theory directly. 

As a first step in this section we show how $V^*(I)$ can be identified with 
a variety. For 
this purpose we consider the list $x$ as a matrix and thus define $det(x)$ 
as a polynomial in $K[x]$. Let $x_0$ be another indeterminate over $K$ and 
write $\hat{x} = (x_0, x_1, \ldots, x_{n^2})$. Define $f_0 = x_0 det(x) -1 
\in K[\hat{x}]$ and 
\[ \hat{I} = I + (f_0) \unlhd K[\hat{x}].\]

We write the elements of $V(\hat{I})$ as $(v_0, v)$ with $v_0 \in K$ and
$v \in K^{n \times n}$ and thus we consider $V(\hat{I})$ as subset of 
$K \oplus K^{n \times n}$. We define a componentwise multiplication on 
$K \oplus K^{n \times n}$ via 
\[ (v_0, v) (u_0, u) = (v_0 u_0, vu). \]

\begin{lemma}
\label{identify}
Let $K$ be an arbitrary field and $I \unlhd K[x]$.
\begin{items}
\item[\rm (a)]
The projection $\zeta : K \oplus K^{n \times n} \ra K^{n \times n} : 
(v_0, v) \ms v$ induces an bijection between $V(\hat{I})$ and $V^*(I)$.
\item[\rm (b)]
$V^*(I)$ is closed under multiplication (inversion) if and only if
$V(\hat{I})$ is closed under multiplication (inversion).
\end{items}
\end{lemma}

\begin{proof}
(a) Let $(v_0, v), (w_0, w) \in V(\hat{I})$ with $\zeta((v_0,v)) =
\zeta((w_0,w))$. Then $v=w$ and thus $v_0 = det(v)^{-1} = det(w)^{-1}
= w_0$. Hence $(v_0, v) = (w_0, w)$ and $\zeta$ is injective on 
$V(\hat(I))$. By the construction of $\hat{I}$, the image of $\zeta$
coincides with the invertible elements in $V(I)$ and thus with 
$V^*(I)$. \\
(b) The map $\zeta$ is compatible with multiplication and thus 
(b) follows directly from (a). 
\end{proof}

As a second step in this section we exhibit how one can readily decide 
if $V(I)$ and $V^*(I)$ are equal. The following lemma is elementary and
can be proved readily using Hilbert's Nullstellensatz.

\begin{lemma}
\label{identical}
Let $K$ be an algebraically closed field, let $I \unlhd K[x]$ and 
$J = (I, \det(x)) \unlhd K[x]$.
\begin{items}
\item[\rm (a)] 
Then $V(J) \subseteq V(I)$ and $V^*(I)$ is the set-complement to
$V(J)$ in $V(I)$.
\item[\rm (b)]
$V(I) = V^*(I)$ if and only if $1 \in J$.
\end{items}
\end{lemma}

Note that both cases $V(I) = V^*(I)$ and $V(I) \neq V^*(I)$ can occur.
For example, if $I = \ideal{\det(x)-1}$, then $V(I) = V^*(I) = SL(n,K)$.
In contrast,
if $I = \ideal{0}$, then $V(I) = K^{n \times n}$ while $V^*(I) = GL(n,K)$,
and thus $V(I) \neq V^*(I)$.

%%%%%%%%%%%%%%%%%%%%%%%%%%%%%%%%%%%%%%%%%%%%%%%%%%%%%%%%%%%%%%%%%%%%%%%%%%%%%
\section{Closedness of $V^*(I)$ under inversion}
\label{secinv}

In this section we show how to decide if $V^*(I)$ is closed under 
inversion.  By definition, every element in $V^*(I)$ is invertible 
in $GL(n,K)$.  It remains to check if the inverses are contained 
in $V^*(I)$. 

Consider the list of indeterminates $x$ as a matrix in $K[x]^{n \times n}$.
Then there exists a formal inverse $N(x)$ for the matrix $x$. The matrix
$N(x)$ is an element of $K(x)^{n \times n}$ and it can be determined via 
the classical adjoint. Each entry of $N(x)$ has the form $g_{ij}(x)/det(x)$ 
for a polynomial $g_{ij}(x) \in K[x]$.

\begin{theorem} \label{inv}
Let $K$ be an algebraically closed field and let $I = (f_1, \ldots, f_r) 
\unlhd K[x]$. For $1 \leq i \leq r$ there exist $h_i(x) \in K[x]$ and 
$l_i \in \N_0$ with $f_i(N(x)) = h_i(x)/det(x)^{l_i}$. Write $k_i(x) = 
h_i(x) det(x) \in K[x]$. Then the following are equivalent.
\begin{items}
\item[\rm (A)]
$V^*(I)$ is closed under inversion.
\item[\rm (B)]
$h_i(v) = 0$ for $1 \leq i \leq r$ and for every $v \in V^*(I)$.
\item[\rm (C)]
$k_i(v) = 0$ for $1 \leq i \leq r$ and for every $v \in V(I)$.
\item[\rm (D)]
$k_i \in \sqrt{I}$ for $1 \leq i \leq r$.
\item[\rm (E)]
$h_i \in \sqrt{\hat{I}}$ for $1 \leq i \leq r$.
\end{items}
\end{theorem}

\begin{proof}
Let $J = (k_1, \ldots, k_r) \unlhd K[x]$. Then Condition (C) is equivalent 
to $V(I) \subseteq V(J)$ and this, in turn, is equivalent to $J \subseteq 
\sqrt{I}$ by Hilbert's Nullstellensatz. Hence (C) and (D) are equivalent. 
Similarly, (B) and (E) are equivalent via Lemma \ref{identify}.

\smallskip
\noindent
(A) $\Ra$ (B):
Let $v \in V^*(I)$, so $\det(v) \neq 0$.  Put $w = v^{-1}$.
By construction $w = N(v)$, and by (A) we have $w \in V^*(I)$.
Hence for each $i=1,\ldots,r$ we have $0 = f_i(w) = f_i(N(v)) = 
h_i(v)/\det(v)^{l_i}$, whence $h_i(v) = 0$.

\smallskip
\noindent
(B) $\Ra$ (C):
Let $i \in \{1, \ldots, r\}$ and $v \in V(I)$.  Then either $v \in V^*(I)$
and thus $h_i(v) = 0$ by (B), or $\det(v) = 0$. In either case, it follows 
that $k_i(v) = \det(v) \, h_i(v) = 0$ as desired.

\smallskip
\noindent
(C) $\Ra$ (A):
Let $v \in V^*(I)$, so $\det(v) \neq 0$.  Now, $w = v^{-1}$ exists as 
an element of $GL(n,K)$, and $w = N(v)$.  By (C) it follows that $0 = 
k_i(v) = \det(v) \, h_i(v) = \det(v) \, f_i(N(v)) = \det(v) \, f_i(w)$ 
for each $1 \leq i \leq r$. Whence each $f_i(w) = 0$ for $1 \leq i \leq r$.
Thus $w \in V(I)$. As $w$ is invertible, it follows that $w \in V^*(I)$. 
\end{proof}

Both, Theorem \ref{inv} (D) and (E) translate directly to an algorithm for 
checking if every element of $V^*(I)$ is invertible inside $V^*(I)$. 

%%%%%%%%%%%%%%%%%%%%%%%%%%%%%%%%%%%%%%%%%%%%%%%%%%%%%%%%%%%%%%%%%%%%%%%%%%%%%
\section{Closedness of $V^*(I)$ under multiplication}
\label{secmul1}

Now we introduce a method to decide if $V^*(I)$ is closed under multiplication.
Recall that $\hat{x} = (x_0, \ldots, x_{n^2})$ and let $\hat{y} = (y_0, 
\ldots, y_{n^2})$. Define $\varphi : K[\hat{x}] \ra K[\hat{y}] : x_i \ms
y_i$ and
\[ \hat{I}_{xy} = \hat{I} + \varphi(\hat{I}) \unlhd K[\hat{x},\hat{y}].\]

\begin{theorem} \label{mult1}
Let $K$ be an algebraically closed field. Then the following are 
equivalent.
\begin{items}
\item[\rm (A)]
$V^*(I)$ is closed under multiplication.
\item[\rm (B)]
$f_i(vw) = 0$ for $1 \leq i \leq r$ and for every $v,w \in V^*(I)$.
\item[\rm (C)]
$f_i(xy) \in \sqrt{\hat{I}_{xy}}$ for $1 \leq i \leq r$.
\end{items}
\end{theorem}

\begin{proof}
Note that (C) is equivalent to $V(\hat{I}_{xy}) \subseteq V(J)$ with
$J = (f_1(xy), \ldots, f_r(xy)) \unlhd K[\hat{x}, \hat{y}]$ by Hilbert's 
Nullstellensatz. Using this equivalence it is easy to show that the 
three statements are equivalent.
\end{proof}

%%%%%%%%%%%%%%%%%%%%%%%%%%%%%%%%%%%%%%%%%%%%%%%%%%%%%%%%%%%%%%%%%%%%%%%%%%%%%
\section{Deciding if $V^*(I)$ is a group}

Let $1_n$ denote the $n \times n$ identity matrix. The set $V^*(I)$ is a 
group if  $1_n \in V^*(I)$, $V^*(I)$ is closed under inversion and $V^*(I)$
is closed under multiplication. This can checked now in the following steps.
Let $K$ be an algebraically closed field and recall that $I = (f_1, \ldots,
f_r)$.
\bigskip

{\bf Algorithm 'IsGroup'} \\
For $1 \leq i \leq r$ do
\begin{items}
\item[(1)]
Check that $f_i(1_n) = 0$; if not, then return false.
\item[(2)]
Check that $f_i(N(x)) \in \sqrt{\hat{I}}$; if not, then return false.
\item[(3)]
Check that $f_i(xy) \in \sqrt{\hat{I}_{xy}}$; if not, then return false.
\end{items}
\bigskip

Note that checking membership in a radical of an ideal can be done with
the trick of Rabinowitch and does not require to determine the radical
explicitly.

We mention a variation of the above algorithm which is based on the
following theorem. Its proof is similar to the proof of Theorem
\ref{mult1} and we omit it here.

\begin{theorem} \label{mult2}
Let $K$ be an algebraically closed field and suppose that the identity
matrix is contained in $V(I)$. Then the following are equivalent.
\begin{items}
\item[\rm (A)]
$V^*(I)$ is a group.
\item[\rm (B)]
$f_i(vw^{-1}) = 0$ for $1 \leq i \leq r$ and for every $v,w \in V^*(I)$.
\item[\rm (C)]
$f_i(x N(y)) \in \sqrt{\hat{I}_{xy}}$ for $1 \leq i \leq r$.
\end{items}
\end{theorem}

Hence the above algorithm has the following variation.
\bigskip

{\bf Algorithm 'IsGroup'} \\
For $1 \leq i \leq r$ do
\begin{items}
\item[(1)]
Check that $f_i(1_n) = 0$; if not, then return false.
\item[(2)]
Check that $f_i(x N(y)) \in \sqrt{\hat{I}_{xy}}$; if not, then return false.
\end{items}
\bigskip

There are further variations of this method possible in special cases.
For example, if $V(I) = V^*(I)$, then it is not necessary to use the
ideal $\hat{I}$, but $I$ can be used directly instead.

%%%%%%%%%%%%%%%%%%%%%%%%%%%%%%%%%%%%%%%%%%%%%%%%%%%%%%%%%%%%%%%%%%%%%%%%%%%%%
\section{Examples}
\label{examples}

If $V(I)$ is closed under multiplication, then $V^*(I)$ is also closed 
under multiplication.  However, the converse does not necessarily hold, 
as the following example shows.

\begin{exam}
Let $n = 2$ and consider $f_1 = x_3$ and $f_2 = x_2 (x_2 x_4 - 1)$
and $f_3 = x_1 x_2$.
Then 
\[ V(I) = \left\{ 
    \left( \begin{array}{cc} a & 0 \\ 0 & b \end{array} \right) \mid
    a, b \in K \right\}
\quad\bigcup\quad
\left\{ 
    \left( \begin{array}{cc} 0 & c \\ 0 & c^{-1} \end{array} \right)  \mid
    c \in K \setminus \{0\} \right\} \]
and
\[ V^*(I) = \left\{ 
    \left( \begin{array}{cc} a & 0 \\ 0 & b \end{array} \right) \mid
    a, b \in K \setminus \{0\} \right\}. \]
Thus $V^*(I)$ is closed under multiplication, while $V(I)$ is not.
\end{exam}

Our approach via ideals and varieties requires to work over an algebraically
closed field $K$. In various applications it would be of interest to consider 
$V(I) \cap GL(n, k)$ where $k$ is not necessarily algebraically closed.
Clearly, if $V(I)$ is a group, then $V(I) \cap GL(n, k)$ is also a group.  
The converse is not true, as the following example shows.

\begin{exam}
With $n = 1$, consider $f_1 = (x_1 -1)(x_1^2-2)$ and let $I = 
\ideal{f_1} \,\unlhd\, \C[x_1]$.  Then 
\begin{items}
\bulit 
$V(I) = \{ (1), (\sqrt{2}), (-\sqrt{2})\}$ which is not a group.
\bulit 
$V(I) \cap GL(n, \Q) = \{ (1) \}$ which is a group.
\end{items}
\end{exam}

If $k$ is a finite field (of cardinality $q$) then we can restrict to
$V(I) \cap GL(n, k)$ by adding the polynomial conditions $x_i^q-x_i=0$
for each $i=1,\ldots,n^2$. 

\begin{exam}
Let $\F_q$ denote the field with $q$ elements and let $f_1 = (x_1 -1)(x_1^2-2)
\in \F_5[x_1]$. For $t \in \N_0$ Consider 
\[ I_t = (f_1, x_1^{5^t}-x_1) \unlhd \F_5[x_1].\]
Then $V(I_t) = \{ (1) \} \subset \F_{5^t}$ if $t$ is odd, and $V(I_t) = 
\{ (1), (\sqrt{2}), (-\sqrt{2}) \} \subset \F_{5^t}$ if $t$ is even.  
Thus $f_1$ defines a group iff we restrict to a field of cardinality 
$5^t$ with $t$ odd.
\end{exam}

Further we exhibit some sample applications of our methods. We used the
CoCoA implementation of our methods to check if $V^*(I)$ is a group in
these examples.

\begin{exam}
Choose $n = 3$.
\begin{items}
\item[\rm (1)]
Consider 
\begin{eqnarray*}
f_1 &=&
  850x_1-475x_2-50x_3+1496x_4-836x_5-88x_6+238x_7-133x_8-14x_9,  \\
f_2 &=&
  125x_1-75x_2+25x_3+220x_4-132x_5+44x_6+35x_7-21x_8+7x_9 
\end{eqnarray*}
and let $I = (f_1, f_2)$. Then $V^*(I)$ is a group.
\item[\rm (2)]
Consider
\begin{eqnarray*}
f_1 &=& -3*x_1+x_3-9*x_7+3*x_9, \\
f_2 &=& 52*x_1-16*x_3+169*x_7-52*x_9, \\
f_3 &=& 3*x_4-x_6,
\end{eqnarray*}
and let $I = (f_1, f_2, f_3)$. Then $1_3 \in V(I)$ holds and 
$V(I) \neq V^*(I)$ can be readily observed. Further, $V^*(I)$ is not
a group, as it is not closed under inverses.
\item[\rm (3)]
Consider
\begin{eqnarray*}
f_1 &=& 22*x_1+77*x_2-6*x_4-21*x_5+48*x_7+168*x_8, \\
f_2 &=& 2*x_7+7*x_8, \\
f_3 &=& -14*x_1-49*x_2+4*x_4+14*x_5-28*x_7-98*x_8 \\
\end{eqnarray*}
and let $I = (f_1, f_2, f_3)$.  Then $1_3 \not \in V(I)$
and hence $V^*(I)$ is not a group.
\end{items}
\end{exam}

\begin{exam}
Choose $n = 2$ and consider
\begin{eqnarray*}
f_1 &=& (x_1-1)(x_1^2 + 1),\\
f_2 &=& x_2, \\
f_3 &=& x_3 \\
f_4 &=& x_4 - 1 
\end{eqnarray*}
and let $I = (f_1, f_2, f_3, f_4)$.  Then $V(I)$ is finite
and can be listed explicitly as
\[ V(I) = \left\{
  \left( \begin{array}{cc} 1 & 0 \\ 0 & 1 \end{array} \right), 
  \left( \begin{array}{cc} i & 0 \\ 0 & 1 \end{array} \right), 
  \left( \begin{array}{cc} -i & 0 \\ 0 & 1 \end{array} \right) \right\}. \]
Hence $1_2 \in V(I)$ and each element of $V(I)$ is invertible in $V(I)$
so that $V(I) = V^*(I)$ holds.  But $V(I)$ is not a group, as it is not 
closed under multiplication.
\end{exam}

%%%%%%%%%%%%%%%%%%%%%%%%%%%%%%%%%%%%%%%%%%%%%%%%%%%%%%%%%%%%%%%%%%%%%%%%%%%%%
\section{Comments}

We have implemented the algorithms described in the Section above in
CoCoA-5~\cite{CoCoA}; the implementations are publicly available in 
the CoCoA-5 package \texttt{ArithGroup.cpkg5} (distributed starting 
from CoCoA version 5.1.2).

The implementations are largely straightforward.  We observe that the
programs do not need to compute the radical of an ideal (a costly
operation). Our methods entail checking whether the homomorphic images of
certain polynomials lie in the radical of the ideal $\hat{I}$.
Observe that it suffices to compute a polynomial equivalent (modulo
$\hat{I}$) to the actual homomorphic image; in other words we may employ
``normal form'' reductions modulo $\hat{I}$ when convenient to lower
the overall cost of evaluating $f_i(h)$ modulo $\hat{I}$.

%%%%%%%%%%%%%%%%%%%%%%%%%%%%%%%%%%%%%%%%%%%%%%%%%%%%%%%%%%%%%%%%%%%%%%%%%%%%%
\def\cprime{$'$}

\end{document}